\documentclass[reqno,10pt,notitlepage]{amsart}
\usepackage{amsmath,amsfonts,amsthm,amssymb}
\usepackage{enumerate}
\usepackage{indentfirst}
\usepackage[dvips]{graphicx}
\usepackage[all]{xy}

\newcommand{\Ocal}{\mathcal{O}}

\newcommand{\Fcal}{\mathcal{F}}
\newcommand{\Gcal}{\mathcal{G}}

\newcommand{\wt}{\widetilde}

\renewcommand{\P}{\mathbb{P}}

\newcommand{\C}{\mathbb{C}}

\newcommand{\euler}{e}

\newcommand{\Op}{\Ocal_p}

\DeclareMathOperator{\I}{I}
\DeclareMathOperator{\sing}{Sing}
\DeclareMathOperator{\codim}{codim}

\newtheorem{theorem}{Theorem}[section]
\newtheorem{lemma}[theorem]{Lemma}

\newtheorem{proposition}[theorem]{Proposition}

\newtheorem{thm}[theorem]{Theorem}

\newtheorem{cor}[theorem]{Corollary}

\theoremstyle{definition}

\newtheorem{remark}[theorem]{Remark}
\newtheorem{example}[theorem]{Example}

\theoremstyle{plain}

\font\smallrm=cmr8

\begin{document}

\author[{\smallrm{}T.~Fassarella and N.~Medeiros}]
{Thiago Fassarella \and Nivaldo Medeiros\\
}

\thanks{The authors are partially supported by 
FAPERJ, Proc. E-26111.849/2010 and E-26102.769/2008.}

\title[{\smallrm{}On the polar degree of hypersurfaces}]
{On the polar degree of projective hypersurfaces}

\begin{abstract}
Given a hypersurface in the complex projective $n$-space we prove several known formulas for the
degree of its polar map by purely algebro-geometric methods. Furthermore, we
give formulas for the degree of its polar map in terms  of the degrees of the polar maps of its components.
As an application, we classify the plane curves with polar map of low degree, 
including a very simple proof of I. Dolgachev's classification of homaloidal plane
curves.
\end{abstract}
\maketitle

\section{Introduction}

Let $f$ be a homogeneous polynomial in $n+1$ variables defined over the field of complex numbers.
In this note we analyze the
relationship
between the geometry of its zero scheme $D\subset\P^n$  and properties of its
\emph{polar map} $\nabla f\colon\P^n\dashrightarrow \P^n$, given by its partial derivatives.
Easy cases are well understood: the hypersurface
$D$ is smooth if and only if the polar map is a morphism;
and $D$ has a non-vanishing Hessian if and only if
$\nabla f$ is dominant.
Nevertheless, there are classical questions still
waiting for a satisfactory answer. For example, one asks for a classification of \emph{homaloidal hypersurfaces},
that is, those whose polar map is birational. This is a natural problem 
which has received a lot of attention recently.
We give a quick survey of its status.

A basic result towards the classification is that the answer depends
only on the topology of its zero set. To be precise, if we let $d_{t}(D)$ or
$d_t(f)$ denote the \emph{polar degree} of $f$, defined as the
topological degree of its polar map, then
$d_t(f)=d_t(f_{\mathrm{red}})$ (\cite[Cor.~2]{DP}, \cite{FP}). So,
for classification purposes, we may assume all hypersurfaces are
reduced.

The plane case has been settled by I.~Dolgachev and the complete
list is quite neat \cite[Thm.~4]{Do}: a reduced homaloidal plane
curve must be either the union of three non-concurrent lines; or a
smooth conic; or the union of a smooth conic and a tangent line. On
the other hand, the picture for higher dimensions is completely
different: as it has been shown by C.~Ciliberto, F.~Russo and
A.~Simis, there are irreducible homaloidal hypersurfaces of any
degree $\geq 2n-3$ for every $n\geq 3$ \cite[Thm.~3.13]{CRS}.
Interestingly enough, all these examples exhibit a common feature,
to wit, a complicated singular locus, usually non-reduced. Actually,
A. Dimca conjectured that do not exist homaloidal hypersurfaces of
degree at least three with only \emph{isolated singularities} whenever $n\geq 3$.
There are some results giving plausibility to this, such as
\cite[Thm.~9]{Di},  \cite[Prop.~3.6]{CRS} and \cite[Cor.~3.5]{Ah}.

A bit more ambitiously, one may ask for formulas for the polar
degree. And as we shall see in the sequel, there are plenty. For
starters, we associate to a hypersurface $D$ of degree $k$ a foliation $\Fcal$, now in
$\P^{n+1}$, simply by taking the pencil generated by $f$ and $x_{n+1}^k$. Next,
consider the \emph{Gauss map} $\Gcal_\Fcal\colon
\P^{n+1}\dashrightarrow  \check{\P}^{n+1}$ of the foliation,
$p\mapsto T_p\Fcal$. The upshot is that the maps $\nabla f$ and
$\Gcal_\Fcal$ have the same degree, that is,
\begin{equation}
\label{eq:00}
 d_t(D) = d_t(\Gcal_\Fcal).
\end{equation}
This has already been shown in \cite{FP}. With hindsight, we realized that all polar degree formulas known to us can be derived from \eqref{eq:00} by purely algebro-geometric methods, often with simpler proofs.
One of the aims of this note is to
show how this can be done.

Furthermore, it would be interesting to have formulas expressing the
polar degree of a hypersurface in terms of that of its components,
for this may help to reduce the classification problem to
irreducible hypersurfaces. Such formulas actually do exist, as we
shall prove below.

\medskip
Let us describe the contents of this paper. We state our main results along the way.

\medskip
In Section~2 we outline the basic theory of holomorphic foliations needed for the
our purposes. We give a simple proof of \eqref{eq:00} and from it we prove in
Proposition~\ref{P:milnor} the first formula for the polar degree: 
\emph{given a reduced hypersurface $D\subset \P^n$ of degree $k$ with only isolated singularities},
\begin{equation}
\label{eq:01}
\textstyle
d_t(D) = (k-1)^n - \sum \mu_p
\end{equation}
where $\mu_p$ is the \emph{Milnor number} of the singularity.

\medskip
Section~3 is devoted to the classification of plane polar maps with low degree.
From identity~\eqref{eq:01} and elementary properties of the Milnor number we
obtain in Theorem~\ref{thm:PF1} a formula for
the polar degree in terms of its subcurves:
\emph{given reduced curves $C_1,C_2\subset \P^2$ with
no common components,}
\begin{equation}
\label{eq:02}
d_t(C_1\cup C_2) = d_t(C_1) + d_t(C_2) + \#(C_1\cap C_2)-1.
\end{equation}
We emphasize that the intersection points in \eqref{eq:02} are
counted \emph{without} multiplicities. This seems surprising at
first sight, but ties up with the general philosophy that the polar degree
depends more on the topology than on the algebra. From this Dolgachev's classification of homaloidal plane curves follows
easily, see Theorem~\ref{thm:deg1}. Next, we
list all reduced curves with polar degree two
(exactly 9 types) and three (exactly 31 types), see Theorems~\ref{thm:deg2}
and \ref{thm:deg3}.

\medskip
Back to the general case, we start Section~4 by showing that the
polar degree can also be given as the top Chern
class of the bundle of logarithmic differentials of a resolution of $D$;
precisely, we prove in Proposition~\ref{prop:Cherntop}:
\emph{let $\pi\colon X \to \P^n$ be an
embedded resolution of singularities of $D$. Then, for a generic hyperplane $H\subset\P^n$,}
\begin{equation}
\label{eq:03}
\textstyle
 d_t(D) = \int_X c(\Omega^1_X(\log \pi^*(D+H))).
\end{equation}
Again, this can be derived directly from \eqref{eq:00}, as it has been already noted
in \cite{FP}. Now, by taking into account that there is a ``Gauss-Bonnet theorem
for a complement of a divisor'' (see Remark~\ref{rmk:GB}), we obtain in
Corollary~\ref{P:dt2} an algebraic proof of a nice formula
by Dimca and Papadima \cite[Thm.~1]{DP}
relating the polar degree and the topological Euler characteristic
of the complement of a generic hyperplane section:
\begin{equation}
\label{eq:05}
d_t(D) = (-1)^n (1-\euler(D\setminus H)).
\end{equation}
This allows one to compute the polar degree
effectively, thanks to P.~Aluffi's algorithm for the computation of
the Euler characteristic via Chern-Schwartz-MacPherson classes \cite{Alu03}, implemented in \textsc{Macaulay2} \cite{M2}.

With identity \eqref{eq:05} at hand, a straightforward use of inclusion-exclusion
principle for the Euler characteristic yields a generalization of \eqref{eq:02}:
\emph{given $D_1,D_2$ any hypersurfaces in $\P^n$},
\begin{equation}
\label{eq:06}
 d_t(D_1\cup D_2) = d_t(D_1) + d_t(D_2) + (-1)^n(\euler(D_1\cap D_2\setminus H)-1).
\end{equation}
In fact, we have more general versions of formulas
\eqref{eq:05} and \eqref{eq:06}, concerning the \emph{projective 
degrees} of the polar map; see Corollary~\ref{P:dt2} for more
details. Notice that \eqref{eq:02} follows immediately from \eqref{eq:06}, but
our first proof is more elementary.

\medskip
Finally, in Section~5, we give some applications.
Building on examples given in \cite[Thm.~3.13]{CRS}
and inspired by formula~\eqref{eq:06},
we show in Example~\ref{dtiguais} there are reduced homaloidal hypersurfaces in $\P^n$ of any degree $\geq 2$ for every $n\geq3$,
thus eliminating the bound $2n-3$ mentioned above.
There is a catch, though: in contrast with theirs, our examples
are reducible.

The polar degree
of hypersurfaces in $\P^n$ with normal crossings is computed in Example~\ref{ex:normalcrossings}; by specializing to hyperplanes,
we compute in Example~\ref{hyperplanes} the projective degrees of the standard Cremona transformation of $\P^n$ for any value of $n$, recovering a result proved by G.~Gonzalez-Sprinberg and I.~Pan in \cite[Thm.~2]{GP}.

Our last application is given in Example~\ref{ex:Bruno}. 
Once more helped by \eqref{eq:06}, we present a somewhat short proof of one of the 
main results of \cite{Bruno}: \emph{a collection of \emph{$r$} 
distinct hyperplanes in $\P^n$ is homaloidal if and only if $r=n+1$ and they are in 
general position}.
 
\bigskip
\noindent\textsc{Acknowledgments}.
We owe a great deal to Jorge V. Pereira for sharing many of his ideas and for
pointing out some helpful references. 
We thank Giuseppe Borrelli, Eduardo Esteves,
Marco Pacini, Ivan Pan  and Israel Vainsencher for valuable discussions
on the subject.

\section{From foliations to polar maps}
A \emph{codimension one singular holomorphic foliation}, from now on
just a foliation, $\mathcal F$ on $\mathbb P^n$ is
determined by a line bundle $\mathcal L$ and an element $\omega \in
\mathrm H^0(\mathbb P^n, \Omega^1_{\mathbb P^n} \otimes \mathcal L)$ satisfying

\begin{enumerate}
\item[(i)] $\codim \sing(\omega) \ge 2$ where $\sing(\omega)
= \{ x \in \mathbb P^n \mid \omega(x) = 0 \}$;
\item[(ii)] $\omega \wedge d\omega =0$ in $\mathrm H^0(\mathbb P^n,\Omega^3_{\mathbb P^n} \otimes \mathcal
L^{\otimes 2}).$
\end{enumerate}

The \emph{singular set} of $\mathcal F$, for short $\sing(\mathcal
F)$, is by definition equal to  $\sing(\omega)$. By Frobenius' Theorem, the integrability
condition (ii) determines in an analytic neighborhood of every point
$p \in \mathbb P^n \setminus \sing(\mathcal F)$ a holomorphic fibration of codimension one with
relative tangent sheaf coinciding with the subsheaf of $T\mathbb P^n$
determined by the kernel of $\omega$. Analytic continuation of the
fibers of this fibration describes the leaves of $\mathcal F$.

One of the most basic invariants attached to an isolated singularity of a foliation $\mathcal F$ is its
\emph{multiplicity} $\mu_p(\mathcal F)$, defined as the intersection
multiplicity at $p$ of the zero section of $\Omega^1_{\mathbb P^n}\otimes
\mathcal L$ with the graph of $\omega$. Thus, if $\omega_p=\sum_{i=1}^{n}a_idx_i$ is a local $1$--form defining $\mathcal F$ in a neighborhood of $p$, then
\[
\mu_p(\mathcal F)=\dim_\C\frac{\mathcal O_p}{(a_1,\dots,a_n)}.
\]

The \emph{degree} of  a foliation of $\mathbb P^n$ is geometrically defined as
the number of tangencies of $\mathcal F$ with a generic line $\P^{1}
\subset \mathbb P^n$. If $\iota\colon \P^{1} \to \mathbb P^n$ is the
inclusion of such a line, then the degree of $\mathcal F$ is the
degree of the zero divisor of the twisted $1$-form  $\iota^*\omega
\in \mathrm H^0(\mathbb \P^{1}, \Omega^1_{\P^{1}} \otimes \mathcal L_{|
\P^{1}})$. Since $\Omega^1_{\P^{1}}=\mathcal O_{\P^{1}}(-2)$  the degree of $\mathcal F$ is just $\deg(\mathcal L)-2$.

\medskip

It follows from  Euler sequence that a $1$-form $\omega \in
\mathrm H^0(\mathbb P^n ,\Omega_{\P^{n}}^1 ( \deg(\mathcal F) + 2) )$ can be
interpreted as  a homogeneous $1$-form on $\mathbb C^{n+1}$, still
denoted by $\omega$,
\[
\omega = \sum_{i=0}^n A_i dx_i
\]
with the $A_i$ being homogeneous polynomials of degree
$\deg(\mathcal F) + 1$ and satisfying Euler's relation $ i_R \omega
= 0,$ where  $i_R$ stands for the interior product with the radial vector field 
$R = \sum_{i=0}^n x_i \frac{\partial}{\partial x_i}$.

The \emph{Gauss map} of a foliation $\mathcal F$ of $\mathbb P^n$ is
the rational map
\begin{eqnarray*}
\mathcal G_{\mathcal F} \colon \mathbb P^n &\dashrightarrow& \check
{\mathbb P}^n \,  \\
p &\mapsto& T_p \mathcal F
\end{eqnarray*}
where $T_p \mathcal F$ is the tangent space of the leaf
of $\mathcal F$ through $p$. If we interpret $(dx_0: \cdots : dx_n)$ as projective coordinates of
$\check{\mathbb P}^n$, then the Gauss map of the foliation $\Fcal$ is just the rational map $p\mapsto \left( A_{0}(p):\cdots : A_{n}(p)\right).$

Let $H=\iota(\P^{n-1})\subset \P^{n}$ be a hyperplane given by the inclusion $\iota\colon \P^{n-1}\to \P^{n}$. 
If $\iota^{*}(\omega)$ is identically zero, we say that $H$ is invariant by $\mathcal F$;
otherwise, after dividing the $1$--form $\iota^{*}(\omega)$ by a codimension one singular set if necessary, 
we consider the restriction $\mathcal F_{|H}$ as the foliation defined by this $1$--form. The following well-known lemma (cf. \cite{CLS}), 
which follows from Sard's Theorem applied to  $\mathcal G_{\mathcal F}$, will 
be useful to obtain some information for  the topological degree of $\mathcal G_{\mathcal F}$.

\begin{lemma}\label{L:CLS}
If $H \subset \mathbb P^n$ is a generic hyperplane and $\mathcal F$
is a foliation on $\mathbb P^n$, then the degree of
$\mathcal F_{|H}$ is equal to the degree of $\mathcal F$ and, moreover,
\[
\sing(\mathcal F_{|H}) = (\sing(\mathcal F)\cap H ) \cup \mathcal
G_{\mathcal F}^{-1}(H)
\]
with $ \mathcal G_{\mathcal F}^{-1}(H)$ being finite and all the
corresponding singularities of $\mathcal F_{|H}$ have multiplicity
one.
\end{lemma}

Now let $D\subset \P^n$ be a hypersurface given
by a homogeneous polynomial $f \in \C[x_0,\dots,x_n]$ of degree $k$.
We denote by $d_t(D)$ or $d_t(f)$ its \emph{polar degree}, defined as
the topological degree of its \emph{polar map}
\begin{eqnarray*}
\nabla f\colon  \P^n &  \dashrightarrow & \P^n\\
 x &\mapsto&  (f_{x_0}(x):\cdots:f_{x_n}(x)).
\end{eqnarray*}
Since the polar degree
depends only on the zero locus of $f$ (cf. \cite{DP,FP}) we may suppose $f$ reduced.

We associate to this hypersurface a foliation $\mathcal F$ in $\mathbb P^{n+1}$
defined by the pencil generated by $f$ and $x_{n+1}^k$, that is, the foliation induced by the $1$-form
\[
\omega = x_{n+1}df - kfdx_{n+1}.
\]

\begin{remark}\label{R:polardegree}
Let $\mathcal G_{\mathcal F}$ be the Gauss map associated to this foliation
\begin{eqnarray*}
\mathcal G_{\mathcal F} \colon  \P^{n+1} &  \dashrightarrow & \P^{n+1}\\
 x &\mapsto&  (x_{n+1}f_{x_0}(x):\cdots:x_{n+1}f_{x_n}(x):-kf(x)).
\end{eqnarray*}
If  $\rho(x_0: \cdots : x_n : x_{n+1}) = (x_0 : \cdots : x_n)$ is the projection with center at $p=(0:\cdots:0:1)$, we see that the rational maps $\mathcal G_{\mathcal F}$ and $\nabla
 f$ fit  in the commutative diagram below.
\[
\xymatrix { \mathbb P^{n+1} \ar@{-->}[d]_{\rho}
\ar@{-->}[rrr]^{\mathcal G_{\mathcal F}} &&& {\mathbb P}^{n+1}
 \ar@{-->}[d]_{{\rho}}\\
\mathbb P^n \ar@{-->}[rrr]^{\nabla f} &&&
{\mathbb P}^n}
\]
A simple computation shows that the restriction of $\rho$ to a fiber of $\mathcal G_{\mathcal F}$ induces an isomorphism to the corresponding fiber of $\nabla f$ and so
their topological degrees  coincide, that is, $d_t(f)=d_t(\mathcal G_{\mathcal F})$.
This is a particular case of  \cite[Thm.~2]{FP} where higher degrees are also considered.
\end{remark}

With this at hand we are able to recover a formula for the polar
degree of hypersurfaces with only isolated singularities.
The main invariant to be considered here is the \emph{Milnor number}
\[
\mu_p(f)=\dim_\C\frac{\mathcal O_p}{(\tilde{f}_{x_1},\dotsc,\tilde{f}_{x_n})}
\]
where $\tilde{f}$ is
the germ of $f$ at $p$.

\begin{proposition}
\label{P:milnor}
Let  $D\subset \P^n$ be a hypersurface with isolated singularities, 
given by a reduced polynomial $f \in \C[x_0,\dots,x_n]$ of degree $k$. Then
\[
\textstyle
 d_t(D)=(k-1)^n - \sum \mu_p(f).
\]
\end{proposition}
\begin{proof}
Consider the foliation $\mathcal F$ on $\mathbb P^{n+1}$ induced by the $1$-form
\[
\omega = x_{n+1}df - kfdx_{n+1}.
\]
Notice that all the singularities of this foliation are contained in $D$.

It follows from Lemma \ref{L:CLS} that the
degree $d_t(\mathcal G_{\mathcal F})$ of the Gauss map is given by the number of
isolated singularities of $\mathcal F|_{H}$ away from $\sing(\mathcal F)$, where $H$ is generic hyperplane on $\P^{n+1}$.

Denote by $h=x_{n+1}|_H$ and $\tilde{f}=f|_H$ the restrictions to $H$.
Thus $\mathcal F|_H$ is induced by the $1$--form
\[
\eta = hd\tilde{f} - k\tilde{f}dh.
\]

Let us suppose ${h=\sum_{i=1}^{n}a_ix_i}$ with $a_0\neq 0$. On the one hand the
singular set of $\eta$ outside $Z(h)$ is given by $\mathcal G_{\mathcal F}^{-1}(H) \cup \sing{D}$;
and on the other hand it is also given by the intersection of $n$ hypersurfaces of degree $k-1$
\[
\bigcap_{i=1}^{n}Z(a_0 f_{x_i} - a_i f_{x_0})
\]
so by B\'ezout's Theorem we get
\(
d_t(\mathcal G_{\mathcal F}) + \sum \mu_p(f) = (k-1)^n.
\)
Now the proposition follows from Remark~\ref{R:polardegree}.
\end{proof}

For different proofs see \cite[p.\thinspace487]{DP} and \cite[Example~11]{Huh}.

\section{Plane polar maps of low degree}

The main result of this section is a formula for computing the
polar degree of a plane curve in terms of that
of its components. Once we have established that, we present
the classification of all reduced plane curves with polar degree
less or equal than three.

\begin{thm}\
\label{thm:PF1}
\begin{enumerate}[{\rm1.}]
 \item Given an irreducible curve $C\subset\P^2$ of degree $k$,
 then
\begin{equation}
\label{dt:1}
\textstyle
d_t(C) = k-1 + 2p_g + \sum (r_p-1)
\end{equation}
where $p_g$ is the geometric genus and $r_p$ is the
number of branches at $p$.
\item
Given two reduced curves $C,D$ in $\P^2$ with no
common components, we have
\begin{equation}
\label{PF}
d_t(C\cup D) = d_t(C) + d_t(D) + \#(C\cap D)-1.
\end{equation}
\end{enumerate}
\end{thm}
\begin{proof} A more general version of \eqref{dt:1} already appeared in \cite{Do} but we give the argument for the
reader's convenience.
Since $C$ is irreducible, the genus formula gives
\[
 \textstyle
    p_g= (k-1)(k-2)/2 - \sum \delta_p
\]
where $\delta_p=\dim_\C \wt{\Op}/\Op$ is the codimension of the local ring $\Op$ in its normalization.
Now combine this with the \emph{Milnor Formula}
\[
\mu_p = 2\delta_p - r_p+1
\]
and Proposition~\ref{P:milnor} to get the result.

Let's prove \eqref{PF}.
Write $C=Z(f)$, $D=Z(g)$ and let $k,l$ be their degrees.
By Proposition~\ref{P:milnor}, the polar degree of the product $fg$ is
\(
(k+l-1)^2 - \sum \mu_p(fg).
\)
Rewriting,
\begin{equation}
\label{eq:fg}
d_t(fg) =
(k-1)^2 + (l-1)^2 + 2kl-1 - \!\!\!
\sum_{p\in C\cap D} \!\!\!\mu_p(fg)-\!\!\!\sum_{p\in C\setminus D} \!\!\!\mu_p(f)-\!\!\!\sum_{p\in D\setminus C} \!\!\!\mu_p(g).
\end{equation}
Since $fg$ is reduced, we have the well-known identity
\cite[Cor.~6.4.4]{CA}
\[
    \mu_p(fg)=\mu_p(f)+\mu_p(g) + 2\I_p(f,g) - 1
\]
where $\I_p$ is the intersection multiplicity. Plugging this into
\eqref{eq:fg} and applying Proposition~\ref{P:milnor}
\[
d_t(fg) = d_t(f) + d_t(g) + 2kl - 2\!\!\!\sum_{p\in C\cap D}\!\!\!\I_p(f,g)+ \#(C\cap D) - 1
\]
and hence by B\'{e}zout's Theorem we are done.
\end{proof}

Theorem~\ref{thm:PF1} makes the classification
of homaloidal plane curves amazingly simple,
since the polar degree never decreases
whenever a new component is added.
We are ready to prove the
celebrated Dolgachev's theorem \cite[Thm.~4]{Do}:
\begin{thm}
\label{thm:deg1}
A reduced homaloidal curve in the projective plane
must be one the following:
\begin{enumerate}[{\rm1.}]
    \item Three nonconcurrent lines.
    \item A smooth conic.
    \item A smooth conic and a tangent line.
\end{enumerate}
\end{thm}
\begin{proof}
Let $C\subset\P^2$ be a reduced homaloidal curve. Then it follows from equations
\eqref{dt:1} and \eqref{PF} that $C$ can have only lines  and
at most one irreducible conic as components.  If $C$ has a conic,
then \eqref{PF} imposes $C$ can have at most one line, which must
meet the conic at a single point.

Assume now our curve has only lines.
These cannot be all concurrent,  for otherwise
the polar degree is zero by \eqref{PF}. If $C$ is the union of three nonconcurrent lines,
then $d_t=1$ by \eqref{PF}. Finally, if $C$ has more than
three lines, not all concurrent, once more formula \eqref{PF} gives $d_t>1$.
\end{proof}

The polar degree two case is also quite easy.
\begin{thm}
\label{thm:deg2}
A reduced plane curve with polar degree two
must be one of the following:
\begin{enumerate}[{\rm1.}]
    \item Three concurrent lines and a fourth line not meeting
    the center point.
    \item A smooth conic and a secant line.
        \item A smooth conic, a tangent and a line passing thru the tangency point.
    \item A smooth conic and two tangent lines.
    \item Two smooth conics meeting at a single point.
    \item Two smooth conics meeting at a single point and the common tangent.
    \item An irreducible cuspidal cubic.
    \item An irreducible cuspidal cubic and its tangent at the smooth flex point.
    \item An irreducible cuspidal cubic and its tangent at the cusp.
\end{enumerate}
\begin{figure}[h]
\centering
\includegraphics[width=0.5\textwidth]{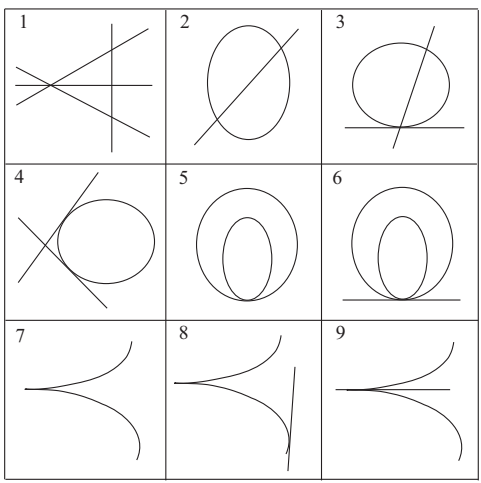}
\caption{Plane curves with $d_t=2$.}
\end{figure}
\end{thm}
\begin{proof}
Such a curve cannot have components
of degree greater than 3. From \eqref{dt:1}  we see
that an irreducible cubic with $d_t=2$ must be cuspidal;
and in view of \eqref{PF} we may attach to it at most one line
and they ought to meet at a single point. This accounts for the last three cases
in the statement.

The remaining cases,
unions of lines and conics, may be analyzed by inspection
and are exactly the ones listed above.
\end{proof}

We proceed to classify the plane curves with polar degree three.
\begin{theorem}
\label{thm:deg3}
A reduced plane curve with polar degree three must be one of the 31 types
shown in Figure~2.
\nopagebreak
\begin{figure}[ht]
\centering
\includegraphics[width=0.8\textwidth]{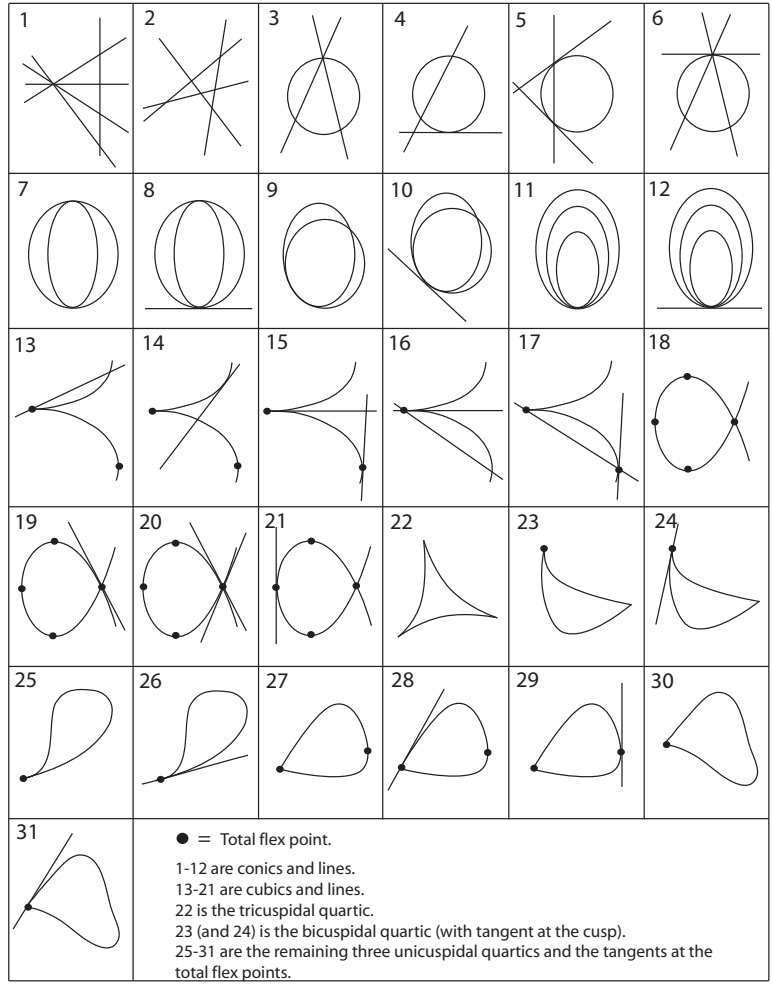}
\caption{Plane curves with polar degree three.}
\label{fig2}
\end{figure}
\end{theorem}
\begin{proof}

Here irreducible quartics will hop in but
only cuspidal rational ones are allowed in view of equation~\eqref{dt:1}.
As before, we may attach to such a quartic a tangent line but
only at \emph{total} flex points.
A neat list of all cuspidal
rational plane quartics and their flex points can be found
in \cite{Moe} (see also \cite{Nam}): up to projective equivalence there are exactly five,  labeled 22,
23, 25, 27 and 30 in Figure~2.

Besides, we have to deal with unions
of cubics, conics and lines. Keeping an eye in \eqref{dt:1} and
\eqref{PF}, all it takes is a careful analysis
of the possible configurations and there is not much to say about
it, except for the case of cubics and conics, which deserves more attention.

Let $C$ be a conic and let $D$ be a cubic, both irreducible.
Since the conic is homaloidal, it follows from
equation \eqref{PF},
\[
d_t(C\cup D) = d_t(D) + \#(C\cap D).
\]
Since smooth and nodal cubics have polar degree $\geq 3$,
we see there is only one possible case, namely, the cubic must be cuspidal and it must intersect the conic at a single point. This case
is missing in Figure~2, for a very simple reason: this configuration does not exist! Although we believe this is well-known we are
obliged to give a proof, for we lack a suitable reference.

Let $D$ be a cuspidal cubic and let us show that $C$ and
$D$ cannot meet at a point with multiplicity six. By tensorizing the exact sequence that defines the ideal sheaf of $D$ by $\mathcal O_{\P^{2}}(2)$, we get
\[
0 \longrightarrow \mathcal O_{\P^{2}}(-1)  \longrightarrow \mathcal O_{\P^{2}}(2) \longrightarrow i_{*}\mathcal O_{D}(2) \longrightarrow 0
\]
and the long exact sequence in cohomology yields an isomorphism
\[
{\mathrm{H^{0}}}(\P^{2}, \mathcal O_{\P^{2}}(2))  \longrightarrow {\mathrm{H^{0}}}(D, \mathcal O_{D}(2)).
\]
Hence, if a conic meets $D$ at a point $p$ with multiplicity six,
this conic is unique. If $p$ is the cusp or the flex, then this conic is the double tangent line.
Assume $p$ is neither the cusp nor the flex. Let $L$
be the tangent line at $p$ and write $L\!\cdot\!D = 2p+q$.
If there were a conic intersecting $D$ at $p$ with multiplicity six, we would get a linear equivalence
\(
4p+2q \sim 6p,
\)
and hence $2p\sim 2q$, in $D\setminus \{c\}$, where $c$ denotes the cusp. Let $2q+r$ the divisor defined by the restriction on the tangent at $q$. Then $2p+q\sim 2q+r$ and thus $q\sim r$ in $D\setminus \{c\}$. This implies that $q=r$ (see \cite[Example~6.11.4]{Har}). Therefore $q$ would be the flex point of $D$. But $D$ has no points $p$ for which the tangent line passes through the flex point, as can be checked by direct
computation or, better, by a simple argument with the dual curve of $D$, which is also a cuspidal cubic.
\end{proof}


\section{Logarithmic differential forms and the Euler characteristic}
In this section we show that the formula \eqref{eq:05} of Dimca and Papadima, which relates the polar degree of a hypersurface with the topological Euler characteristic of its affine part, can also be obtained by algebraic methods. As a consequence, we give  a generalization of equation \eqref{PF} for higher dimensions.

\medskip
Let $X$ be a smooth projective variety of dimension $n$. We say that a divisor $D=\sum D_i$ of $X$ has \emph{normal crossings} if $D$ is reduced, 
each component $D_{i}$ is smooth and at each point
of intersection of some of the divisors $D_i$, say $D_1,\dots,D_k$,
there are local analytic coordinates $z_1,\dots,z_n$ for $X$
so that $D_i$ is given by $z_i=0$ for $i=1,\dots,k$.

Let us suppose $D$ reduced. We define the \emph{sheaf of logarithmic differentials} $\Omega^1_X(\log D)$ as a subsheaf of the sheaf $\Omega^1_X(D)$ of $1$-forms with poles at most on $D$ and of order one. This sheaf is the image of the natural map $\Omega_{X}^{1}\otimes \mathcal O_{X}^{\oplus n} \to \Omega^1_X(D)$, which is given by the inclusion $\Omega_{X}^{1} \to \Omega^1_X(D)$ and by the homomorphisms $\mathcal O_{X} \to \Omega^1_X(D)$ sending $1 \mapsto d\log f_{i}$, where $f_{i}$ is a local equation of $D_{i}$. The reducedness assumption on $D$ is not essential here,
for the differential logarithmic of a power of a function is a multiple constant of the differential logarithmic of the function. So $D$ and $D_{\mathrm{red}}$ define the same sheaf $\Omega^1_X(\log D)$. A basic known fact is that if $D$ has normal crossings, then the sheaf $\Omega^1_X(\log D)$ is locally free of rank $n$.

It has been shown in \cite{FP} that the polar degree can also be given as the degree of the total Chern
class, denoted here by $c(\cdot)$, of the bundle of logarithmic differentials of a resolution of $D$, as we shall review now.

\begin{proposition}
\label{prop:Cherntop}
Let $D\subset\P^n$ be a reduced hypersurface and
$\pi\colon X \to \P^n$ be an
embedded resolution of singularities of $D$ so that
the total transform $\pi^*D$ has normal crossings.
Then, for a generic hyperplane $H\subset\P^n$,
\begin{equation}
\label{eq:Cherntop}
\textstyle
 d_t(D) = \int_X c(\Omega^1_X(\log \pi^*(D+H))).
\end{equation}
\end{proposition}
\begin{proof}
This equality is a consequence of the Remark~\ref{R:polardegree}. We sketch the argument in the next few lines for the convenience of the reader.

Suppose $D$ is given by a homogeneous polynomial $f \in \C[x_0,\dots,x_n]$ of degree~$k$.
Let $\mathcal F$ be the foliation on $\mathbb P^{n+1}$ induced by the rational $1$-form
\[
\omega = \frac{df}{f} - k\frac{dx_{n+1}}{x_{n+1}}
\]
and $\mathcal G_{\mathcal F}$ its Gauss map. By Lemma \ref{L:CLS}, $d_{t}(\mathcal G_{\mathcal F})$ coincides with the number of isolated singularities of $\Fcal|_{\mathbb P^{n}}$ that are not singularities of $\mathcal F$; here, $\mathcal F |_{\mathbb P^{n}}$ is the restriction of $\mathcal F$ to a generic $\mathbb P^{n} \subset \mathbb P^{n+1}$. Since all the singularities of $\mathcal F$ are contained in $Z(x_{n+1} f)$ we have just to count the isolated singularities of  $\mathcal F |_{\P^n}$ away from $Z(x_{n+1} f)$. The intersection in  $\mathbb P^{n+1}$ of  $Z(x_{n+1} f)$ and a generic $\mathbb P^{n}$ is isomorphic to the union of $D$ with a generic hyperplane $H\subset \mathbb P^{n}$. Thus the rational $1$-form $\omega|_{\mathbb P^{n}}$ can be viewed as an element of ${\rm{H^{0}}}(\mathbb P^{n}, \Omega_{\mathbb P^{n}}^{1}(\log (D+H)))$.
Since $\pi$ is an embedded resolution for $D$, Bertini's Theorem implies that it is also an embedded resolution of $D\cup H$ and therefore $\Omega^1_X(\log \pi^{*}(D+H))$ is locally free.  Assuming the singular scheme of $\pi^{*}\omega|_{\mathbb P^{n}} \in \mathrm{H}^{0}(X, \Omega_{X}^{1}(\log \pi^{*}(D+H)))$ has just a zero-dimensional part, its length is measured by the top Chern class of $\Omega^1_X(\log \pi^{*}(D+H))$. And the latter assumption follows from the fact that the residues of $\pi^{*}\omega|_{\mathbb P^{n}}$ on each irreducible component of the support of $\pi^{*}(D+H)$ are non-zero. For details see \cite[Lemmas 3 and 4]{FP}.
\end{proof}

\begin{remark}
\label{rmk:GB}
Let $D$ be a smooth projective variety. 
The topological Euler characteristic of $D$ is computed
by the degree of
the Chern total class of the tangent bundle of $D$, that is,
\[
\textstyle
  \euler(D) = \int_D c(TD).
\]
This is
the Gauss-Bonnet theorem.
For the case that $D$ is a divisor with normal crossings
of a smooth variety $X$, we have the following
version:
\begin{equation}
 \label{eq:04}
\textstyle
\euler(X\setminus D)=(-1)^n \int_X c(\Omega^1_X(\log D)).
\end{equation}
Here, $X\setminus D$ is complement of the support of $D$. This has
been proved by R. Silvotti \cite[Thm.~3.1]{Si} and recovered by P.
Aluffi \cite[\S\thinspace2.2]{Alu99} (along the way of his
characterization of Chern-Schwartz-MacPherson classes), but both
were predated by Y. Norimatsu \cite{No}, a two-page gem that
apparently had fallen into oblivion.
\end{remark}

Let $D=Z(f)\subset \P^n$ be a hypersurface
and consider its polar map $\nabla f\colon\P^n \dashrightarrow  \check{\P}^n$. For each $i=0,\dots,n-1$,
we denote by $d_i(D)$ the \emph{$i$-th projective degree} of the polar map, defined as the degree of the closure of the algebraic set $\nabla f|_{U}^{-1}(L_{i})$, where $L_{i} \subset  \check{\P}^n$ is a generic linear subspace of dimension $i$ 
and $U\subset\P^n$ is the Zariski open set where the polar map is regular. Notice that $d_{0}(D)$ is just the polar degree $d_{t}(D)$. It follows from \cite[Cor.~2]{FP}  that $d_{i}(D)$ coincides with the topological degree of the Gauss map of the restriction of the foliation $\mathcal F$ to a generic $\mathbb P^{n+1-i}\subset \mathbb P^{n+1}$, that is,
\(
d_{i}(D)=d_{t}(\mathcal G_{\mathcal F|_{\mathbb P^{n+1-i}}}).
\)
Since $\mathcal F|_{\P^{n+1-i}}$ is the foliation associated to the divisor $D\cap \mathbb P^{n-i}$ for a generic  $\mathbb P^{n-i}\subset \mathbb P^{n}$, it follows from  Remark~\ref{R:polardegree} that
\begin{equation}
\label{restriction}
d_{i}(D)=d_{t}(D\cap \mathbb P^{n-i}).
\end{equation}

We are ready to give formulas relating the projective degrees of the polar map with the Euler characteristic.

\begin{cor} Let $\P^i, H\subset\P^n$ denote a generic linear subspace of dimension $i$ and a generic hyperplane, respectively. 
\label{P:dt2} 
\begin{enumerate}[{\rm1.}]
\item Given a hypersurface $D\subset\P^n$, we have 
\begin{equation}
\label{dt:2}
 d_i(D)= (-1)^{n-i} ( 1 - \euler(D\cap\P^{n-i}\setminus H) )
\end{equation}
for  $i=0,\dots,n-1$.
\item Given two hypersurfaces $D_1, D_2\subset \P^n$, we have
\begin{equation}
\label{PF2}
d_i(D_1 \cup D_2) = d_i(D_1) + d_i(D_2) +
(-1)^{n-i} ( \euler( D_1\cap D_2\cap \P^{n-i}\setminus H ) - 1 )
\end{equation}
for $i=0,\dots,n-1$.
\end{enumerate}
\end{cor}
\begin{proof} In view of \eqref{restriction}, it is enough to prove both assertions for $i=0$. 

Let us prove that formula \eqref{dt:2} holds for $i=0$. 
We may assume $D$ reduced.
Let $\pi\colon X \to \P^n$ be an embedded resolution of singularities of $D$ such that $\pi^*D$ has
normal crossings. Hence it follows
from \eqref{eq:Cherntop} and \eqref{eq:04} that
\begin{align*}
 (-1)^nd_t(D)
&= \textstyle(-1)^n \int_X c(\Omega^1_X(\log \pi^*(D+H))) \\
&= \euler(X\setminus \pi^*(D + H))\\
&= \euler(\P^n\setminus (D + H))\\
&= 1 - \euler(D\setminus H)
\end{align*}
where in the last equality have used $\euler(\P^n)=n+1$ and the inclusion-exclusion principle for the Euler characteristic  of algebraic varieties. 

Finally, \eqref{PF2} follows from \eqref{dt:2} and 
straightforward manipulations using inclusion-exclusion.
\end{proof}

Dimca and Papadima \cite[Thm.~1]{DP} proved formula~\eqref{dt:2} for $i=0$,
by using topological methods. The generalization for higher values
of $i$ has been proved by Huh \cite[Thm.~8]{Huh} quite recently.
The expression \eqref{dt:2} written here is slightly different
from theirs.  Furthermore, we point out that an identity for the polar degree similar
to \eqref{PF2} has appeared in \cite[Prop.~5]{Di}, but only for
complete intersections. 

\begin{remark}
Let $D_1,D_2 \subset \P^n$ be hypersurfaces with no
common components. It would be interesting to know whether the
`correction term' in \eqref{PF2}:
\[
(-1)^n ( \euler( D_1\cap D_2\setminus H ) - 1 )
\]
is always non-negative,
as this would reduce the
classification of homaloidal hypersurfaces to the irreducible ones.
That is the case in all examples
we have checked. If that were the case in general, then we would have
the inequality
\[
   d_t(D_1\cup D_2) \geq \max \{d_t(D_1), d_t(D_2)\}
\]
proved in \cite[Cor.~3]{FP} using foliations.
\end{remark}

\section{Applications}

\begin{example}
\label{dtiguais}
 Starting with a hypersurface $D$ of degree $d$ in $\P^{n}$,
there is a very simple way to construct another hypersurface $Y$, now
of degree $d+1$ and sitting in $\P^{n+1}$, with the
same polar degree: let $C, L\subset \P^{n+1}$ be
the projective cone over $D$ and a generic hyperplane respectively and take
$Y$ as their union. Since cones and hyperplanes have null polar degree
and $C\cap L$ is isomorphic to $D$, equations~\eqref{PF2} and \eqref{dt:2} yield
\begin{align*}
 d_t(Y) & = d_t(C) + d_t(L) + (-1)^{n+1} (\euler((C\cap L)\setminus H)-1)\\
        & = (-1)^n (1-\euler(D\setminus h)) \\
        & = d_t(D)
\end{align*}
where $H$ and $h$ stand for generic hyperplanes in $\P^{n+1}$ and $\P^n$, respectively.

As a consequence, let us show that there are homaloidal  reduced hypersurfaces in $\P^n$
for all $n\geq 3$ and of any degree $\geq 2$.
Indeed, this is already known for $n=3$ \cite[Thm.~3.13]{CRS} and so our construction gives homaloidal hypersurfaces in $\P^4$
for any degree $\geq 3$. As for degree two, we already know that
smooth quadrics are always homaloidal, so we are done
for $n=4$. Now one has just to workout the same reasoning for higher values of $n$.

The same construction gives
a family of counter-examples for the \emph{Hesse's problem} (see \cite{CRS}), namely, hypersurfaces
with vanishing Hessian that are not cones: one has only to
notice that null Hessian is equivalent to polar degree zero
and that if $D$ is not a cone, then $Y$ is not a cone
as well.
\end{example}

\begin{example}
\label{ex:normalcrossings}
Let  $D=\sum_{i=1}^{r}D_{i}$ be a reduced divisor in $\P^n$ with normal crossings. Let $H\subset\P^n$ be a generic hyperplane.
From the exact sequence of locally free sheaves on $\P^n$
\begin{equation*}
\label{seq1}
    0 \to \Omega_{\P^n}^1 \to \Omega_{\P^n}^1(\log(D+H)) \to \oplus_i \Ocal_{D_i} \oplus \Ocal_H \to 0
\end{equation*}
and from Euler's exact sequence
\[
    0 \to \Omega_{\P^n}^1 \to \Ocal_{\P^n}(-1)^{\oplus n+1} \to \Ocal_{\P^n} \to 0
\]
we get, by Whitney's formula 
and Proposition~\ref{prop:Cherntop}, that
\begin{equation}
\label{eq:normalcrossings}
 {\textstyle d_t(D) = \int_{\P^n} c(\Omega_{\P^n}^1(\log(D+H)))}
 = \left[ \frac{(1-h)^n}{(1-k_1 h)\dotsb (1-k_r h)} \right]_n
\end{equation}
where $[\cdot]_n$ stands for the coefficient of $h^n$ and $k_{i}$ is the degree of $D_{i}$. For $r=2$,
\[
 d_t(D_1\cup D_2) = d_t(D_1)+d_t(D_2) + q
\]
where
\[
q=
 \begin{cases}
  \displaystyle\frac{k_2(k_1-1)^n - k_1(k_2-1)^n}{k_1-k_2} & \text{if $k_1\neq k_2$}\\[0.4cm]
  (k-1)^n(nk+k-3) & \text{if $k=k_1=k_2$}.
 \end{cases}
\]
\smallskip
Here the correction term is always non-negative.
\end{example}

\begin{example}
\label{hyperplanes}
For a collection of $r$ hyperplanes in $\P^n$ in \emph{general position}, a simple calculation
from \eqref{eq:normalcrossings} yields
\begin{equation}
\label{eq:hyperplanes}
 d_t=
\begin{cases}
 0 & \text{if $r \leq n$}\\
 \binom{r-1}{n} & \text{if $r\geq n+1$}
\end{cases}
\end{equation}
so the polar map is birational if and only if $r=n+1$.
In that case the map is, up to change of coordinates, the \emph{standard Cremona transformation} 
$\phi\colon\mathbb P^{n} \dashrightarrow \mathbb P^{n}$, given as the polar map associated to $f=x_{0}\cdots x_{n}$. Now, identities \eqref{restriction} and \eqref{eq:hyperplanes} together yield expressions for the projective degrees of this map, namely
\[
\textstyle
d_{i}(\phi)= \binom{n}{n-i}
\]
for $i=0,\dots,n-1$. These numbers have also been obtained
by G.~Gonzalez-Sprinberg and I.~Pan
\cite[Thm.~2]{GP} by applying methods of toric geometry.
\end{example}

%

\begin{example}
\label{ex:Bruno}
The first result of the last example holds in more generality. In fact, 
A. Bruno \cite[Thm.~A]{Bruno} proved:
\emph{a collection of \emph{$r$} distinct hyperplanes in $\P^n$ is homaloidal if and only if $r=n+1$
and they are in general position.}
We offer an alternative proof, based
on  formula \eqref{PF2}.

Indeed, if $r\le n+1$ and they are not in general position, then the arrangement is a cone, so the polar degree is zero; and we have seen
in Example~\ref{hyperplanes} that $n+1$ hyperplanes in general
position is homaloidal. So, if $D$ is the union of $r$ hyperplanes, it suffices to show that 
 $d_t(D)\geq 2$ whenever $r\geq n+2$ and $D$ is not a cone.  The conclusion goes by induction on $n$. 
For $n=1$ this is immediate, so we assume $n\geq 2$ and $D$ is not a cone. Let $D'$ be the union of the first $r-1$ hyperplanes and $H_r$ be the last one.
By equation~\eqref{PF2}
\[
d_t(D) = d_t(D') + d_t(H_r) +
(-1)^{n} (\euler(D'\cap H_r \setminus H)-1).
\]
Now, $D'\cap H_r$ is an arrangement of $r-1$ hyperplanes in $H_r$. By equation~\eqref{dt:2},
the correction term in the identity above is exactly the polar degree of this arrangement. Finally, notice that $D'\cap H_r$ is not a cone
because $D$ is not, hence by induction
$d_t(D)\geq d_t(D'\cap H_r)\geq 2$, as wished.
\end{example}

\vspace{1cm}

\font\smallsc=cmcsc9
\font\smallsl=cmsl9

\noindent{\smallsc Universidade Federal Fluminense, Instituto de Matem\'{a}tica, Departamento de An\'{a}lise. \\
Rua M\'{a}rio Santos Braga, s/n, Valonguinho, 24020--140 Niter\'{o}i RJ,
Brazil.}

\vskip0.3cm
{\smallsl E-mail address:
\small\verb?tfassarella@id.uff.br?}

\vskip0.3cm

{\smallsl E-mail address: \small\verb?nivaldo@mat.uff.br?}

\end{document}